\documentclass[10pt]{amsart}
\usepackage{mathrsfs, amssymb}

\newcommand{\QQ}{\mathbb{Q}}

\newcommand{\ZZ}{\mathbb{Z}}
\newcommand{\RR}{\mathbb{R}}
\newcommand{\CC}{\mathbb{C}}

\renewcommand{\phi}{\varphi}
\renewcommand{\epsilon}{\varepsilon}

\newcommand{\hft}{\max\{h_{\mathsf{M}_d}(f_t), \log|\operatorname{Norm}\mathfrak{R}_{f_t}|\}}

\newtheorem{theorem}{Theorem}
\newtheorem{lemma}[theorem]{Lemma}

\theoremstyle{definition}
\newtheorem*{remark}{Remark}

\title[Canonical heights for weighted homogeneous families]{Canonical heights and preperiodic points for certain weighted homogeneous families of polynomials}
\author{Patrick Ingram}
\date{\today}
\address{Current address: York University, 4700 Keele St., Toronto, Canada}
\email{pingram@yorku.ca}
\thanks{The author's research is partially supported by Simons Collaboration Grant \#283120.}

\begin{document}
\maketitle
\begin{abstract}
A family $f$ of polynomials over a number field $K$ will be called \emph{weighted homogeneous} if and only if $f_t(z)=F(z^e, t)$ for some binary homogeneous form $F(X, Y)$ and some integer $e\geq 2$. For example, the family $z^d+t$ is weighted homogeneous. We prove a lower bound on the canonical height, of the form
\[\hat{h}_{f_t}(z)\geq \epsilon \hft,\]
for values $z\in K$ which are not preperiodic for $f_t$. Here $\epsilon$ depends only on the number field $K$, the family $f$, and the number of places at which $f_t$ has bad reduction. For suitably generic morphisms $\varphi:\mathbb{P}^1\to \mathbb{P}^1$, we also prove an absolute bound of this form for $t$ in the image of $\varphi$ over $K$ (assuming the $abc$ Conjecture), as well as uniform bounds on the number of preperiodic points (unconditionally).
\end{abstract}


\section{Introduction}

Let $c\in\mathbb{Q}$, and let $f_c(z)=z^2+c$. Poonen conjectured~\cite{poonen} that $f_c$ cannot have a $\QQ$-rational periodic point of exact period $N$, unless $N\leq 3$.  More generally, Morton and Silverman~\cite{ms} have speculated that if $f$ is a rational function of degree $d\geq 2$ defined over a number field $K$, then there should be a bound on the number of $K$-rational preperiodic points of $f$ which depends just on $d$ and $K$. It is reasonable to ask whether there exists even a single one-parameter family of rational functions for which one might prove the uniform boundedness of $K$-rational preperiodic points. So far the best known bounds depend on the number of places at which $f$ has bad reduction~\cite{benedetto, call-goldstine, canci}.

Certainly there are three ready sources of examples. In the first place, the result is fairly easy to prove for isotrivial families. For families of Latt\`{e}s maps the result is highly non-trivial, but reduces to the uniform boundedness result of Merel~\cite{merel}. Finally, if the family is over a base curve $X$ of general type, Faltings' Theorem gives that there are only finitely many $K$-rational points on the base to start with, and so uniform bounds on preperiodic points are unimpressive. Thus the question becomes ``Can one establish uniform boundedness of preperiodic points for even a single non-isotrivial, non-Latt\`{e}s family of rational functions over a rational (or elliptic) base curve?'' The purpose of this note is to present some special cases in which one can.

Now, with $f$ a rational function over $K$, let $\hat{h}_f$ be the associated canonical height, $h_{\mathsf{M}_d}(f)$ the height of the representative point in moduli space (relative to any fixed ample line bundle), and $\mathfrak{R}_f$ the minimal resultant. It is conjectured by Silverman~\cite[Conjecture~4.98, p.~221]{ads} that over a fixed number field $K$, and for fixed $d\geq 2$, there exists an $\epsilon>0$ such that
\[\hat{h}_f(z)\geq \epsilon \max\{h_{\mathsf{M}_d}(f), \log|\operatorname{Norm}(\mathfrak{R}_t)|\},\]
for any $f(z)\in K(z)$ of degree $d$, and any $z\in \mathbb{P}^1(K)$ which is not preperiodic for $f$, a conjecture motivated by a similar conjecture of Lang for elliptic curves (see~\cite[p.~92]{lang} and \cite{hindry-silverman, jsduke}). In the case of a one-parameter family $f$ depending on a parameter $t$, the quantity $h(t)$ is a reasonable stand-in for the right-hand-side.
This conjecture is superficially unrelated to uniform boundedness, but both in some sense reflect a paucity of points of small height. (Compare, for example, the proofs of \cite[Theorem~6.9]{call-goldstine} and \cite[Theorem~1]{ads_lang}.) We will explore these conjectures as well for certain special families of polynomials.

A family $f$ of polynomials (whose specializations will be denoted $f_t$) will be called \emph{homogeneous of weight $e$} if $e\geq 2$ and $f_t(z)=F(z^e, t)$ for some homogeneous binary form $F(X, Y)$ divisible by neither $X$ nor $Y$. A family is simply \emph{weighted homogeneous} if it is homogeneous of weight $e$ for some $e\geq 2$. The prime example of such a family is the unicritical family $z^d+t$, but for instance $z^{6}-3tz^3+t^2$ is another weighted homogeneous family. Although allowing $e=1$ in the definition would make sense, our results all rely on the hypothesis $e\geq 2$. Every such family is isomorphic, over a certain cover of the base, to a family of the form $g_s(z)=sg_1(z)$.

Our first theorem addresses the conjecture of Silverman, and generalizes Theorem~1 of \cite{ads_lang}.

\begin{theorem}\label{th:badprimes}
Let $K$ be a number field, and let $f$ be a weighed homogeneous family of polynomials defined over $K$. Then for any $s\geq 0$ there exists an $\epsilon>0$ (depending on $K$, the family $f$, and the fixed ample class on $\mathsf{M}_d$) such that the following holds: for any $t\in K$ for which $f_t$ has at most $s$ places of bad reduction, and any $z\in K$ which is not preperiodic for $f_t$,
\[\hat{h}_{f_t}(z)\geq \epsilon\hft.\]
\end{theorem}
The proof of Theorem~\ref{th:badprimes}  rests on the pigeon-hole principle. The type of family defined above has the property that at every place $v$ of bad reduction, there are a few disks, very small when the parameter $t$ is large, such that any point outside of these escapes to infinity in a uniform way. In any sufficiently long segment of an orbit, one of three things happes: we must have either a repetition (and so the orbit is preperiodic of bounded size), the points leave the disks at some place and escape to infinity in a way that bounds their canonical height from below, or two distinct points which are in the same disk at every bad place. These disks are very small, in the last case one of these points must have very large height by the standard Diophantine observation that points of small height cannot be too close to one another. This again gives us a lower bound on the heights in the orbit. 

The argument is motivated by that presented in~\cite{jsduke}, with the aforementioned disks playing the role of the components of the N\'{e}ron model in the elliptic curve case. In the dynamical setting, similar arguments have been applied by Benedetto~\cite{benedetto} and Baker~\cite{baker1, baker2} to obtain upper bounds on preperiodic points and lower bounds on the canonical height, and by the author~\cite{drinf} to obtain a result analogous to~\cite{jsduke} in the context of Drinfeld modules. Here we derive stronger results here by greatly restricting the types of rational function under consideration.

Indeed, just as in~\cite{ads_lang}, we can have $\epsilon$ depending only on the number of places at which a certain type of bad reduction occurs. Although we are unable to prove that $\epsilon$ may be made independent of the number of primes of bad reduction, we can establish this for suitably generic base extensions of the family, under additional hypotheses.

Setting
\[N_e=\begin{cases}
5 & \text{if }e=2\\
4 & \text{if }e=3\\
3 & \text{otherwise,}
\end{cases}
\] we say that the cover $\varphi:\mathbb{P}^1\to\mathbb{P}^1$ is \emph{$e$-general} if and only if $\varphi$ has at least $N_e$ distinct poles in $\mathbb{A}^1(\overline{K})$ of order prime to $e$. 

\begin{theorem}\label{th:abc}
Let $K$ be a number field for which the $abc$ Conjecture holds, let $f$ be a family  of polynomials defined over $K$, homogeneous of weight $e$, and let $\varphi:\mathbb{P}^1_K\to\mathbb{P}^1_K$ be an $e$-general covering. Then there exists an $\epsilon>0$ such that for any $t=\varphi(s)$ with $s\in \mathbb{P}^1(K)$,
\[\hat{h}_{f_{t}}(z)\geq \epsilon\hft\]
for all $z\in K$ which are not preperiodic for $f_{t}$.
\end{theorem}

Theorem~\ref{th:abc} is reminiscent of a result of Hindry and Silverman~\cite{hindry-silverman}, which proves Lang's Conjecture for elliptic curves~\cite[p.~92]{lang} under the hypothesis of the $abc$ Conjecture, and indeed the proof is in part motivated by that result.  In~\cite{hindry-silverman} the $abc$ Conjecture is used to show, via Szpiro's Conjecture, that a significant amount of the height of the $j$-invariant comes from places at which the N\'{e}ron model has relatively few components, attenuating the worst case in the pigeonhole argument in Silverman's earlier argument~\cite{jsduke}. Here, too, we use the $abc$ Conjecture to show that a significant contribution to the height of $\varphi(s)$, under the hypotheses of the theorem, comes from places of bad reduction where the local dynamics is nonetheless still very simple.

The inequality in Theorem~\ref{th:abc} follows from  an inequality of the form
\[\hat{h}_{f_t}(z)\geq \epsilon h(t)-C,\]
for all $z\in K$,
where $\epsilon$ and $C$ depend on the hypotheses of the Theorem. In particular, it follows from the proof of Theorem~\ref{th:abc} that the Uniform Boundedness Conjecture of Morton and Silverman holds for sufficiently general base extensions of weighted homogeneous families if one assumes the $abc$ Conjecture, simply because $\epsilon h(t)-C$ will be positive for all but finitely many $t\in K$. It turns out that one can prove this unconditionally.

\begin{theorem}\label{th:dg}
Let $K$ be a number field, let $f$ be a family of polynomials defined over $K$, homogeneous of weight $e$,  and let $\varphi:\mathbb{P}^1_K\to\mathbb{P}^1_K$ be an $e$-general covering. Then there are only finitely many $t\in \mathbb{P}^1$ such that $f_{\varphi(t)}$ has a preperiodic point in $K$. In particular, the Uniform Boundedness Conjecture is true for the base extension of the family $f$ by $\varphi$.
\end{theorem}
Another way of thinking of Benedetto's Theorem~\cite{benedetto} (restricted to the present context) is that it establishes uniform boundedness for the family $f_t$, not over $t\in K$, but over $t\in\bigcup_{|S|\leq N} \mathcal{O}_{K, S}$, for any $N$ (this also follows from the proof of Theorem~\ref{th:badprimes}). Similarly, Theorem~\ref{th:dg} proves uniform boundedness for $t \in \varphi(K)$ for a suitably general $\varphi$. The theorems are somewhat orthogonal, then, each proving uniform boundedness for parameters in some natural, although admittedly very thin, subset of $K$. The novelty of Theorem~\ref{th:dg} lies in the fact that it provides a uniform bound on preperiodic points for some non-trivial family of polynomials which may have arbitrarily many places of bad reduction.

For instance, it follows from Theorem~\ref{th:dg} that uniform boundedness of $\mathbb{Q}$-rational preperiodic points holds for the family
\[f_t(z)= z^2+\frac{1}{1+t^5},\]
as $t\neq -1$ varies over $\QQ$. This example is perhaps unsurprising, since it is clear that $z^2+c$ has a $\mathbb{Q}$-rational preperiodic point only when the denominator of $c$ is a square, and the denominator of $(t^5+1)^{-1}$ is unlikely to be a square very often (although some work is needed to show this when $t$ is not an integer!). Indeed, this is the motivation for the proof of Theorem~\ref{th:dg}, which leans heavily on a result of Darmon and Granville~\cite{darmon-granville}, which itself follows from Faltings' Theorem~\cite{faltings}. Using more specific results of Darmon and Merel~\cite{dm}, we can be more concrete.
\begin{theorem}\label{th:dm}
Let $2\mid d$ and $m\geq 4$, or $3\mid d$ and $m\geq 3$. Then for $t\in \mathbb{Q}$, the polynomial
\[z^d+\frac{1}{1+t^m}\]
has no $\mathbb{Q}$-rational preperiodic points (other than $z=\infty$).
\end{theorem}

Section~\ref{sec:local} gathers some material on Green's functions for weighted homogeneous families over local fields. Section~\ref{sec:badred} is devoted to the proof of Theorem~\ref{th:badprimes}, applying a pigeon-hole argument at the places of bad reduction, while Section~\ref{sec:abc} contains the proof of Theorem~\ref{th:abc}. Finally, Section~\ref{sec:darmon} proves Theorems~\ref{th:dg} and~\ref{th:dm}, employing results of Darmon-Granville~\cite{darmon-granville} and Darmon-Merel~\cite{dm}.


\section{Local considerations}\label{sec:local}

Let $K$ be a field with a set $M_K$ of  pairwise non-equivalent absolute values, for instance a number field with the usual set of absolute values, or $\CC$ with just the usual absolute value. We will, once and for all, fix an extension of each element of $M_K$ to the algebraic closure $\overline{K}$, and we will normalize non-archimedean valuations $v$ so that $v(K^*)=\ZZ$ (i.e., $v_p(z)=-\log_p|z|_p$ for $K=\QQ$). An \emph{$M_K$-constant} will be a real-valued function $v\mapsto \mathfrak{a}_v$ on $M_K$ which vanishes at all but finitely many valuations. It will be convenient, for much of the below, to set for $r>0$
\[r_v=\begin{cases} r & \text{if }v\text{ is archimedean}\\ 1 & \text{ otherwise,}\end{cases}\]
for instance so that $|x+y|_v\leq 2_v\max\{|x|_v, |y|_v\}$ for any valuation $v$. Note that $\log r_v$ is an $M_K$-constant, and that the set of $M_K$-constants is a vector space over $\RR$.

As usual, we write
\[\log^+ X =\log\max\{1, X\}\]
for $X\in \RR$.
For any $v\in M_K$ and $f(z)\in K[z]$ of degree $d\geq 2$, we let
\[G_{f, v}(P)=\lim_{n\to\infty}d^{-n}\log^+|f^n(P)|_v.\]
Note that, when $K$ is a number field and $M_K$ is the usual set of absolute values, the canonical height associated to $f$ is a weighted sum of the $G_{f, v}$.
We will assume throughout that $f$ is monic, an assumption which we will justify in the next section.
\begin{lemma}\label{lem:greens}
For a monic homogeneous family $f$ of degree $d$ and weight $e$, the following holds, with all constants independent of $t$ (but depending on the coefficients of the generic fibre).
\begin{enumerate}
\item \label{it:basin}There exist $M_K$-constants $\mathfrak{a}$ and $\mathfrak{b}$ such that for all $t, z\in K$ with \[\log|z|_v>\frac{1}{e}\log^+|t|_v+\mathfrak{a}_v\] we have
$|f_t(z)|_v\geq|z|_v$
and
\[\left|G_{f_t, v}(z)-\log|z|_v\right|\leq \mathfrak{b}_v.\]
Moreover, $\mathfrak{b}_v=0$ unless $v$ is archimedean.
\item\label{it:typeicontributes} There exist a finite set $S\subseteq M_K$, depending only on $K$ and the family $f$, such that for any $v\not\in S$ with $e\nmid v(t)$, we have
\begin{equation}\label{eq:Glower}G_{f_t, v}(z)\geq \frac{1}{d}\log^+|t|_v\end{equation}
for all $t, z\in K$.
\item\label{it:filled} There exists a finite set $S\subseteq M_K$, depending only on $K$ and the family $f$, such that  for any $t\in K$ and $v\not\in S$, if $|t|_v>1$ then any preperiodic point $z$ of $f_t$ satisfies $|z|^e_v=|t|_v$.
\end{enumerate}
\end{lemma}

\begin{proof}
By hypothesis, we have
\[f_t(z)=\prod_{i=1}^{d/e}(z^e-\beta_i t).\]
For all $z$ with $|z|_v^e>3_v\max\{1, |\beta_it|_v\}$, we have
\begin{equation}\label{eq:induction}|z|_v\leq |z|^d_v\left(\frac{2}{3}\right)_v^{d/e}\leq |f_t(z)|_v\leq |z|_v^d\left(\frac{4}{3}\right)_v^{d/e}\end{equation}
by the triangle and ultrametric inequalities. Taking
\[\mathfrak{a}_v = \frac{1}{e}\max\log^+|\beta_i|_v+\frac{1}{e}\log 3_v,\]
then, gives the first part of~\eqref{it:basin}. Now~\eqref{eq:induction} allows us to use induction, obtaining 
\begin{multline*}\log|z|_v+ \frac{1}{e}\left(1+d^{-1}+\cdots+d^{-(n-1)}\right)\log((2/3)_v)\\\leq d^{-n}\log |f^n_t(z)|_v\\\leq \log|z|_v+ \frac{1}{e}\left(1+d^{-1}+\cdots+d^{-(n-1)}\right)\log((4/3)_v)\end{multline*}
under the same assumption on $|z|_v$. Taking limits gives~\eqref{it:basin}.

The set $S$ in~\eqref{it:typeicontributes}  will consist of those places for which $\mathfrak{a}_v\neq 0$ or $\mathfrak{b}_v\neq 0$. Since~\eqref{eq:Glower} is immediate if $|t|_v\leq 1$, we will assume that $|t|_v>1$. Note that since $v(z)$ is an integer, and $e\nmid v(t)$, we must have $|z|^e_v\neq |t|_v$. If  $|z|_v^e>|t|_v$, then by~\eqref{it:basin} we have
\[G_{f_t, v}(z)=\log|z|_v\geq \frac{1}{e}\log^+|t|_v\geq \frac{1}{d}\log^+|t|_v.\]
Otherwise $|z|_v^e<|t|_v$. In this case $|f(z)|_v^e = |t|_v^{d}>|t|_v$, and hence
\[dG_{f_t, v}(z)=G_{f_t, v}(f_t(z))=\log|f_t(z)|\geq \log^+|t|_v,\]

Now \eqref{it:filled} follows from~\eqref{it:typeicontributes}, since any preperiodic point $z\in K$ must have $G_{f_t, v}(z)=0$ for all $v\in M_K$.
\end{proof}

Given a monic polynomial $f$ as above, we write
\[g_{f, v}(x, y)=-\log|x-y|_v+G_{f, v}(x)+G_{f, v}(y).\]
for the usual dynamical $v$-adic Arakelov-Green's function relative to $f$, as in \cite[Section~10.2]{br} and~\cite{baker1}. In terms of the construction in~\cite{br}, this definition corresponds to the lift $(x, 1)$ of $[x:1]\in \mathbb{A}^1\subseteq \mathbb{P}^1$ and $F(X, Y)=(Y^d f(X/Y), Y^d)$ of $f$, noting that the resultant term in the usual definition vanishes here under the hypothesis that $f$ is monic. We note that if $f$ has good reduction at $v$, then $g_{f, v}$ is simply the restriction of the logarithmic chordal metric on $\mathbb{P}^1$ to $\mathbb{A}^1$.

\begin{lemma}\label{lem:goodred}
Let $\mathfrak{a}$ and $\mathfrak{b}$ be the $M_K$-constants provided by Lemma~\ref{lem:greens}\eqref{it:basin}. Then for any $v\in M_K$ and $t\in K$ with $|t|_v\leq 1$, and any two distinct points $x, y\in K$, we have
\[g_{f_t, v}(x, y)\geq -\max\{\mathfrak{a}_v, \mathfrak{b}_v\}-\log 2_v.\]
\end{lemma}

\begin{proof}
Without loss of generality assume that $|x|_v\geq |y|_v$, so that $|x-y|_v\leq 2_v|x|_v$.

First, suppose that $\log|x|_v>\mathfrak{a}_v$. Then we have
\begin{eqnarray*}
g_{f_t, v}(x, y)&=&-\log|x-y|_v+G_{f_t, v}(x)+G_{f_t, v}(y)\\
&\geq & -\log|x|_v-\log 2_v+\log|x|_v-\mathfrak{b}_v\\&&+G_{f_t, v}(y)\\
&\geq & -(\mathfrak{b}_v+\log 2_v),
\end{eqnarray*}
by the non-negativity of $G_{f_t, v}$.

If, on the other hand,  $ \log|y|_v\leq \log|x|_v\leq \mathfrak{a}_v$, then we have
\[g_v(x, y)\geq -( \mathfrak{a}_v+\log 2_v)\]
directly from the non-negativity of $G_{f_t, v}$.
\end{proof}

\begin{lemma}\label{lem:pigeon}
 Assume that $d>e$, and again that the family $f$ is monic. There exists a $\delta>0$ and an $M_K$-constant $\mathfrak{c}$ depending only on the  family $f$ and field $K$ 
such that 
for any $x\in K$, any place $v\in M_K$, and any set $X$ of at least $d+3$ non-negative integers, one of the following conditions obtains:
\begin{enumerate}
\item \label{it:repeats} there exist distinct $i, j\in X$ with $f^i_t(x)=f^j_t(x)$, or
\item \label{it:pairsareclose}   there exists a subset $Y\subset X$  such that $\# Y\geq \left(\frac{1}{d+2}\right)\# X$ and for all distinct $i, j\in Y$,
\[g_{f_t, v}\left(f^i_t(x), f_t^j(x)\right)\geq \delta\log^+|t|_v-\mathfrak{c}_v.\]
\end{enumerate}
\end{lemma}

\begin{proof}
We prove the lemma with
\[\delta=\min\left\{\frac{1}{e}, 1-\frac{1}{e}-\frac{1}{d}\right\}.\]

First note that if $|t|_v\leq 1$ and case~\eqref{it:repeats} does not occur, then we may apply Lemma~\ref{lem:goodred} and deduce that item~\eqref{it:pairsareclose} happens for $Y=X$.

From this point forward, assume that $|t|_v>1$ and that case~\ref{it:repeats} does not occur. Note that $\log^+|t|_v>0$.

Let $\mathfrak{a}$ and $\mathfrak{b}$ be the $M_K$-constants from Lemma~\ref{lem:greens}\eqref{it:basin}. Note that if  \begin{equation}\label{eq:fjbig}\log|f^j(z)|_v>\frac{1}{e}\log^+|t|_v+\mathfrak{a}_v,\end{equation} for any $j$,
then $\log|f^{j+1}(z)|_v\geq \log|f^j(z)|_v$,
by  Lemma~\ref{lem:greens}\eqref{it:basin}.  By induction,
\[|f^i_t(z)|_v\geq|f^j_t(z)|_v>\frac{1}{e}\log^+|t|_v+\mathfrak{a}_v\]
for every $i>j$, and hence
\begin{eqnarray*}
g_{f_t, v}(f^i_t(x), f^j_t(x))&=& -\log|f_t^i(x)-f_t^j(x)|_v+G_{f_t, v}(f^i(x))+G_{f_t, v}(f^j(x)) \\
&\geq& -\log|f_t^i(x)|_v-\log 2_v + \log|f_t^i(x)|_v\\&&+\log|f_t^j(x)|_v-2\mathfrak{b}_v\\
&\geq & \frac{1}{e}\log^+|t|_v+\mathfrak{a}_v-\log 2_v-2\mathfrak{b}_v.
\end{eqnarray*}
Let $Y_1\subseteq X$ be the (terminal) subset of indices $j$ for which~\eqref{eq:fjbig} holds. Note that, since~\eqref{eq:fjbig} for $j=k$ implies the same for $j=k+1$, $Y_1$ is a terminal subset of $X$, in the sense that $Y_1=X\cap [J, \infty]$ for some $J$.
If we have $\#X \leq (d+2)\# Y_1$, then we make take $Y=Y_1$ and the proof is complete, since $\delta\leq \frac{1}{e}$, as long as we take $\mathfrak{c}_v\geq \log 2_v+2\mathfrak{b}_v-\mathfrak{a}_v$.

We now assume that $\#X>(d+2)\# Y_1$, and let $X_1=X\setminus Y_1$ so that $\#X_1>\left(\frac{d+1}{d+2}\right)\#X$. For the purpose of this proof, re-write $f_t(z)$ as
\[f_t(z)=\prod (z^e-\beta_i t)^{\alpha_i},\]
with the $\beta_i$ distinct.  Then there is an $M_K$-constant $\mathfrak{e}$ such that
\[\log|f_1(z)|_v\geq \min_{i, \zeta^e=\beta_i} \alpha_i\log|z-\zeta|_v-\mathfrak{e}_v,\]
by the Mean Value Theorem. Write $t=u^e$ (note that $|u|_v>1$). We then have
\[\log|f_t(z)|_v\geq d\log^+|u|_v + \min_{i, \zeta^e=\beta_i}\alpha_i\log\left|\frac{z}{u}-\zeta\right|_v-\mathfrak{e}_v.\]
 Consequently, if $|f(z)|_v\leq \frac{1}{e}\log^+|t|_v+\mathfrak{a}_v$, then we have
\begin{equation}\label{eq:closeto}\alpha_i\log\left|\frac{z}{u}-\zeta\right|_v<(1-d)\log^+|u|_v+\mathfrak{a}_v+\mathfrak{e}_v,\end{equation}
for some $i$ and some $\zeta^e=\beta_i$.
If two points $z_1, z_2$ satisfy~\eqref{eq:closeto} for the same $\zeta$, we then have
\[\log\left|\frac{z_1}{u}-\frac{z_2}{u}\right|_v\leq \frac{1-d}{\alpha_i}\log^+|u|_v+\frac{\mathfrak{a}_v+\mathfrak{e}_v}{\alpha_i}+\log 2_v,\]
and hence
\[\log\left|z_1-z_2\right|_v\leq \frac{1-d+\alpha_i}{\alpha_i}\log^+|u|_v+\frac{\mathfrak{a}_v+\mathfrak{e}_v}{\alpha_i}+\log 2_v,\]
It follows that
\[g_{f_t, v}(z_1, z_2)\geq \frac{d-1-\alpha_i}{e\alpha_i}\log^+|t|_v-\left(\frac{\mathfrak{a}_v+\mathfrak{e}_v}{\alpha_i}+\log 2_v\right)\]
by the non-negativity of $G_{f_t, v}$. Note that $e\alpha_i\leq d$, and so we have
\[g_{f_t, v}(z_1, z_2)\geq \left(1-\frac{1}{e}-\frac{1}{d}\right)\log^+|t|_v-\left(\frac{\mathfrak{a}_v+\mathfrak{e}_v}{\alpha_i}+\log 2_v\right).\]

Now, we partition $X_1$ into no more than $d+1$ sets. First, let $W$ be the set of indices $j\in X_1$ for which $\log|f^{j+1}(x)|v>\frac{1}{e}\log^+|t|_v+\mathfrak{a}_v$. Note that this is at most a single value, since $j\in W$ would imply $i\in Y_1$, for any $i>j$. Now, for each $i$ and each $\zeta^e=\beta_i$ we let $W_\zeta$ be the set of indices $j$ such that~\eqref{eq:closeto} is satisfied by $f^j(z)$. Since there are at most $d$ such values $\zeta$, and since $W$ is a singleton, there must be a set $W_\zeta$ containing at least $\#X_1/(d+1)\geq\#X/(d+2)$ elements. This completes the proof, as long as we take \[\mathfrak{c}_v\geq  \max\{0, \mathfrak{a}_v+\mathfrak{e}_v\}+\log 2_v\quad\text{ and }\quad\delta\leq 1-\frac{1}{e}-\frac{1}{d}.\]
\end{proof}


\section{Canonical height lower bound in places of bad reduction}\label{sec:badred}

The statements in the introduction involve  $\hft$, but it is much simpler in the present context to work in terms of $h(t)$. We first demonstrate that this is an acceptable proxy.

\begin{lemma}\label{lem:useht}
For a weighted homogeneous family $f$,
\[\hft\ll h(t)\]
with implied constants depending on the family.
\end{lemma}

\begin{proof}
The relation
\[h_{\mathsf{M}_d}(f_t)\ll h(t)\]
follows from standard facts about heights. Indeed, $t\mapsto f_t$ induces a morphism $\mathbb{A}^1\to\mathsf{M}_d$, whose image is a curve $\Gamma$ (since the family is not isotrivial, but in any case the inequality is trivial otherwise). If $\widetilde{\Gamma}$ is the closure of the normalization of $\Gamma$, $\psi:\mathbb{P}^1\to \widetilde{\Gamma}$ is the completion of the aforementioned map $\mathbb{A}^1\to \Gamma$, and $D$ is the class obtained by restricting the ample class on $\mathsf{M}_d$ relative to which one is computing heights, we have
\[\deg(\psi^*D)h(t)= h_{\mathbb{P}^1, \psi^*D}(t)+O(1) =h_{\widetilde{\Gamma}, D}(f_t)+O(1)=h_{\mathsf{M}_d}(f_t)+O(1)\]
from the functoriality of height, where the implied constant depends just on the family.
(Note that, in a slight abuse of notation, we are using $f_t$ here to denote the point in $\mathsf{M}_d$ corresponding to the polynomial $f_t$.)

It remains to show that $\log|\operatorname{Norm}\mathfrak{R}_{f_t}|\ll h(t)$, for which it suffices to show that $\log|\operatorname{Norm}\operatorname{Res}(f_t)|\ll h(t)$. Here $\mathfrak{R}_{f_t}$ denotes the minimal resultant, and $\operatorname{Res}(f_t)$ the resultant of this particular model \cite[p.~220]{ads}.

Let $\mathfrak{p}$ be a prime of $K$ and $\mathcal{O}_\mathfrak{p}$ be the ring of $\mathfrak{p}$-integers. Note that for any monic polyomial $g(z)\in K[z]$
\[g(z)=\frac{a_dz^d+\cdots +a_1 z +a_0}{a_d},\]
written with $a_i\in \mathcal{O}_{\mathfrak{p}}$, we have $\operatorname{Res}(g)\mathcal{O}_{\mathfrak{p}}\mid a_d^{2d}\mathcal{O}_{\mathfrak{p}}$ (with equality if $a_i\in \mathcal{O}_{\mathfrak{p}}^\times$ for some $i$). In particular,  we may write
\begin{equation}\label{eq:genericf}f_t(z)=\frac{a_dz^d+\cdots+a_e z^et^{\frac{d}{e}-1} +a_0t^{\frac{d}{e}}}{a_d}\end{equation}
with $a_i$ algebaic integers independent of $t$, and conclude that $v_{\mathfrak{p}}(\operatorname{Res}(f_t))\leq 2dv_{\mathfrak{p}}(a_d)$ whenever  $t\in \mathcal{O}_{\mathfrak{p}}$.

If, on the other hand, $t\not\in\mathcal{O}_\mathfrak{p}$, write $t=t_0/\pi^m$ with $\pi$ a uniformizer and $t_0$ a unit. Then if $f_t$ is written as in~\eqref{eq:genericf}, the representation
\[f_t(z)=\frac{\pi^{dm/e}a_dz^d+\cdots +a_e\pi^{m} z^et_0^{(d/e)-1}+a_0t_0^{d/e}}{\pi^{dm/e}a_d}\]
has integral coefficients. We thus have
\[v_{\mathfrak{p}}(\operatorname{Res}(f_t))\leq v_{\mathfrak{p}}\left((\pi^{dm/e}a_d)^{2d}\right)= \frac{2d^2}{e}v_{\mathfrak{p}}(t^{-1})+2dv_\mathfrak{p}(a_d).\]
Summing over all places, we have
\[\log|\operatorname{Norm}\mathfrak{R}_{f_t}|\leq \log|\operatorname{Norm}\operatorname{Res}(f_t)|\leq [K:\QQ]\frac{2d^2}{e}h(t)+[K:\QQ]2dh(a_d),\]
an explicit form of the desired estimate.
\end{proof}

It is now a relatively simple matter to prove Theorem~\ref{th:badprimes}.

\begin{proof}[Proof of Theorem~\ref{th:badprimes}]
First, we note that we may assume, without loss of generality, that our family of polynomials is monic. In particular, our aim in light of Lemma~\ref{lem:useht} is to prove that
\[\hat{h}_{f_t}(z)\geq \epsilon h(t)-C\]
for all $z, t\in K$ such that $z$ is not preperiodic for $f_t$, and with $\epsilon>0$ and $C$ depending just on the number of valuations $v$ of $K$ with $v(t)<0$. Note that if we prove the claim after a finite extension $L/K$, then the claim is true over $K$ as well, since the heights are well-defined on $\overline{K}$, and since above each place of $K$ their lie between 1 and $[L:K]$ places of $L$. Our hypothese on the form of $f$ imply that
\[f_t(z)=a_dz^d+\cdots + a_0t^{d/e},\]
where $a_d\in K$ does not depend on $t$. Making a finite extension if necessary, write $a_d=\alpha^{d-1}$, and set $g_t(z)=\alpha f_t(\alpha^{-1}z)$, which is monic. If the result is established for the family $g$, then it follows for $f$ simply by noting that $\hat{h}_{f_t}(z)=\hat{h}_{g_t}(\alpha z)$ for all $t, z\in \overline{K}$. So for the rest of the proof we assume that our family is monic.

We now treat the case $d=e$, wherein our family must have the form $f_t(z)=z^d+bt$ for some nonzero $b\in K$. Note that $h(t) = h(bt)+O(1)$, and the number of places at which $t$ is not integral differs by at most a constant from the number of places at which $bt$ is not integral, so the result follows immediately from the main result of~\cite{ads_lang}.

Now assume that $d>e$, and let $f$ be a monic family as described in the statement of the theorem. Let $S=\{v_1, v_2, ..., v_{\#S}\}$ be a set of places containing all of those at which $|t|_v>1$ along with all archimedean places, and suppose that $t\in K$ is integral at all places $v\not\in S$. Let
\[X_0=\left\{0, 1, 2, ..., 2(d+2)^{\# S}\right\}.\]
For any given $x\in K$, if $f^i(x)=f^j(x)$ for some distinct $i,j\in X_0$, then we are done. If not, for each $1\leq i\leq \# S$ choose a subset $X_i\subseteq X_{i-1}$ as per Lemma~\ref{lem:pigeon}\eqref{it:pairsareclose}. We are left with a subset $X_{\# S}$, containing at least two distinct values $i, j$ and such that
\[g_{f_t, v}(f^i(x), f^j(x))\geq \delta\log^+|t|_v-\mathfrak{d}_v\]
for any archimedean $v$, and any $v$ with $|t|_v>1$. Applying Lemma~\ref{lem:goodred} in the case of places with $|t|_v\leq 1$, we have
\[g_{f_t, v}(f^i(x), f^j(x))\geq \delta\log^+|t|_v-\mathfrak{d}_v\]
for all $v$ (taking $\mathfrak{d}_v=\max\{\mathfrak{a}_v, \mathfrak{b}_v\}+\log 2_v$ in cases where we apply Lemma~\ref{lem:goodred}). Summing over all places, we have
\[2d^{2(d+2)^{\# S}}\hat{h}_{f_t}(x)\geq \hat{h}_{f_t}(f^i(x))+\hat{h}_{f_t}(f^j(x))\geq \delta h(t)-C',\]
with $C'=\sum_{v\in M_K}\frac{[K_v:\QQ_v]}{[K:\QQ]}\mathfrak{d}_v$, from which we deduce
\[\hat{h}_{f_t}(x)\geq \epsilon h(t)-C,\]
with $\epsilon=\delta/2d^{2(d+2)^{\# S}}$, and $C=C'/2d^{2(d+2)^{\# S}}$.

So now we have shown that for any set of places $S$ containing all archimedean places, any $t\in K$ which is integral away from $S$, and any $z\in K$, either $f^i_t(z)=f^j_t(z)$ for some $1\leq i< j\leq 2(d+2)^{\#S}$, or else $\hat{h}_{f_t}(z)\geq \epsilon h(t)-C$, with $\epsilon$ and $C$ as above. The theorem follows by noting that for any $t\in K$ integral at all but at most $s$ primes of $K$, we may apply this argument to any set of places containing these places and the at most $[K:\QQ]$ archimedean places.
\end{proof}

\begin{remark}
Note that the proof of Theorem~\ref{th:badprimes} can be improved so that the constants only depend on the number of places $v$ such that $e\mid v(t)$. To see this, if $T$ is the set $S$ of primes from Lemma~\ref{lem:greens}\eqref{it:typeicontributes}, and $v$ is a place with $e\nmid v(t)$, then we have
\[G_{f_t, v}(z)\geq \frac{1}{d}\log^+|t|_v\]
for all $z$. In particular, if $|z|_v\geq |w|_v$, without loss of generality, we have
\[g_{f_t, v}(z, w)\geq \frac{1}{d}\log^+|t|_v+G_{f_t, v}(z)-\log|z|.\]
Since Lemma~\ref{lem:greens}\eqref{it:basin} certainly implies (recall that $\mathfrak{a}_v\mathfrak{b}_v=0$ for $v\not\in T$) that 
\[G_{f_t, v}(z)\geq \log|z|_v,\]
we then have
\[g_{f_t, v}(z, w)\geq \frac{1}{d}\log^+|t|_v,\]
for \emph{any} $z, w\in K$, given that $e\nmid v(t)$ and $v\not\in T$. So in the proof of Theorem~\ref{th:badprimes}, we need only apply Lemma~\ref{lem:pigeon} to archimedean places and those in $T$. This reduces the dependence of the constants in the statement of Theorem~\ref{th:badprimes} to the number of places $v$ with $v(t)<0$ and $e\mid v(t)$.
\end{remark}


\section{Canonical height lower bound for base extensions}\label{sec:abc}

This section concerns the proof of Theorem~\ref{th:abc} and, to start, we explain the motivation of the argument. It follows from the non-negativity of the Green's functions $G_{f_t, v}$ and Lemma~\ref{lem:greens}~\eqref{it:typeicontributes}  that for any weighted homogeneous family $f_t$ there is a finite set $S$ of places such that
\begin{eqnarray}\label{eq:typeilower}\hat{h}_{f_t}(z)&=& \sum_{v\in M_K}\frac{[K_v:\mathbb{Q}_v]}{[K:\mathbb{Q}]} G_{f_t, v}(z)\nonumber\\
&\geq & \frac{1}{d}\sum_{\substack{v\not\in S \\ e \nmid v(t)}}\frac{[K_v:\mathbb{Q}_v]}{[K:\mathbb{Q}]}\log^+|t|_v\end{eqnarray}
for all $z\in K$. In other words, with finitely many exceptions, primes in the denominator of $t$ occuring to powers not divisible by $e$ contribute non-trivially to the canonical height of every point. On the other hand, one should generally expect that the bulk of the height of $\varphi(s)$, where $\varphi:\mathbb{P}^1\to\mathbb{P}^1$ is sufficiently generic rational function and $s\in \mathbb{P}^1(K)$, should come from primes occuring to small powers. In other words, for $t=\varphi(s)$, one should expect the right-hand-side of~\eqref{eq:typeilower} to constitute some positive proportion of $h(t)$, given some assumptions about $\varphi$. An explicit form of this heuristic follows from the $abc$ Conjecture.

%

We now recall the $abc$ Conjecture, in a form useful for the current setting (see~\cite{elkies}). For any $v\in M_K^0$, we will define $\deg(v)=\log|\operatorname{Norm}_{K/\QQ}(\mathfrak{p})|$, where $\mathfrak{p}$ is the maximal ideal in the ring of $v$-integers.  
Given this, we define for $a\in \mathbb{P}^1$
\[N_a(x)=\sum_{v(x-a)>0}\deg(v),\]
where $x-a$ is interpreted as $1/x$ when $a=\infty$. The $abc$ Conjecture is now the statement that for any $\epsilon>0$ there exists a constant $C_\epsilon$ with
\begin{equation}\label{eq:abc}(1-\epsilon)[K:\mathbb{Q}]h(x)\leq N_0(x)+N_1(x)+N_\infty(x)+C_\epsilon\end{equation}
for all $x\in \mathbb{P}^1(K)\setminus \{0, 1, \infty\}$.
For our purposes, it will be convenient to define, for a finite $S\subseteq M_K$ containing all infinite places,
\[N_{a, S}(x)=\sum_{\substack{v\not\in S\\ v(x-a)>0}}\deg(v),\]
and note that replacing $N_{a}$ with $N_{a, S}$ in \eqref{eq:abc} does not change the content of the conjecture, although of course the constant $C_\epsilon$ now depends on $S$. More generally, the conjecture is equivalent (see~\cite{elkies}) to the claim that for any set $Z$ of points,
\[(\#Z-2-\epsilon)[K:\mathbb{Q}]h(x)\leq \sum_{a\in Z}N_{a, S}(x)+C_\epsilon.\]
One direction is clear by taking $Z=\{0, 1, \infty\}$, while the other direction is deduced in~\cite{elkies} by applying a Belyi map that sends $Z$ to these three points. Inspired in part by \cite{granville}, we will use this form of the conjecture to show that a significant part of the height of $\varphi(t)$ comes from primes in the denomonator occuring to powers not divisible by $e$.

We will also specify local height functions at non-archimedean places as follows: for $a\in \mathbb{A}^1$, we set
\[\lambda_{[a], v}(t)=\frac{1}{\mathfrak{e}(v)}\max\{0, v(t-a)\}\log p_v=\log^+\left|\frac{1}{t-a}\right|_v,\]
where $\mathfrak{e}(v)$ is the ramification index of $v$, and $p_v$ the prime of $\QQ$ below $v$. For $a=\infty$, we replace $t-a$ by $1/t$. Note that we have normalized so that the definition is independent of $K$, in the sense that if $L/K$ is a finite extension and $w\mid v$ is a place of $L$, then $\lambda_{[a], v}$ is the restriction to $K$ of $\lambda_{[a], w}$.
 Note also that, for any $a$ and any set $T$ of places, we have
\begin{eqnarray}\label{eq:nonneg}\sum_{v\in T}[K_v:\QQ_v]\lambda_{[a], v}(t)&\leq& \sum_{v\in M_K} [K_v:\QQ_v]\lambda_{[a], v}(t)\nonumber\\&=& [K:\mathbb{Q}]h\left(\frac{1}{t-a}\right)\\&=&[K:\mathbb{Q}]h\left(t\right)+O_a(1)\nonumber,\end{eqnarray}
 since the $\lambda_{[a], v}$ are all non-negative. 

\begin{proof}[Proof of Theorem~\ref{th:abc}]
We will begin by describing a finite set $S$ of places of $K$.
Write
\[\varphi(t)=\beta\prod_{1\leq i\leq r}(t-a_i)^{-e_i}\]
with $e_i\in \ZZ$ non-zero, and $a_i$ distinct. We will include in $S$ any place $v$ with $v(\beta)\neq 0$, or $v(a_i-a_j)\neq 0$ for any $i\neq j$.
 In particular, for $v\not\in S$ and $x\in K$, we have $v(x-a_i)<0$ for at most one root or pole $a_i$ of $\varphi$. We then enlarge $S$  to contain all archimedean places and the set of primes defined in Lemma~\ref{lem:greens}\eqref{it:typeicontributes} as applied to the family $f$.

We now suppose that we have ordered the $a_i$ so that $e_i\geq 1$ and  $\gcd(e, e_i)=1$ for each $1\leq i\leq N_e$. In other words, the first $N_e$ of the $a_i$ are the poles in the hypothesis of the theorem.
 Given this, let
\begin{equation}\label{eq:l1}L_1(t)=\sum_{i=1}^{N_e}\sum_{\substack{v\not\in S \\ v(t-a_i)>0 \\ e\nmid v(t-a_i)}}[K_v:\QQ_v]\lambda_{[a_i], v}(t)\end{equation}
and
\begin{equation}\label{eq:l2}L_2(t)=\sum_{i=1}^{N_e}\sum_{\substack{v\not\in S \\ v(t-a_i)>0 \\ e\mid v(t-a_i)}}\frac{1}{e}[K_v:\QQ_v]\lambda_{[a_i], v}(t).\end{equation}
Let $\mathfrak{e}(v)$ and $\mathfrak{f}(v)$ denote the ramification and residual degrees of $v\in M_K^0$ in the extension $K/\QQ$, so that $[K_v:\QQ_v]=\mathfrak{f}(v)\mathfrak{e}(v)$. Then if $p_v$ is the prime of $\QQ$ above $v$,  we have
\[[K_v:\QQ_v]\lambda_{[a_i], v} = \mathfrak{f}(v)\max\{0, v(t-a)\}\log p_v = \max\{0, v(t-a)\}\deg(v).\] In particular, the summand in~\eqref{eq:l1} corresponding to the place $v$ is a positive integral multiple of $\deg(v)$, and the same is true of~\eqref{eq:l2}.

So by the $abc$ Conjecture we have for any $\epsilon>0$ 
\begin{eqnarray}\nonumber
 (N_e-2-\epsilon)[K:\mathbb{Q}]h(t) &\leq & \sum_{i=1}^{N_e}N_{a_i, S}(t)+O_{\epsilon}(1)\label{eq:line2}\\
&\leq & L_1(t)+L_2(t)+O_{\epsilon}(1)\label{eq:line3}\\
&=&\frac{1}{e}(L_1(t)+eL_2(t))+\left(1-\frac{1}{e}\right)L_1+O_{\epsilon}(1)\nonumber\\
&\leq &\frac{N_e}{e}[K:\mathbb{Q}]h(t)+\left(1-\frac{1}{e}\right)L_1(t) +O_{\epsilon}(1)\label{eq:line5}.
\end{eqnarray}
Here,   \eqref{eq:line3} follows from the fact that the summands in \eqref{eq:l1} and \eqref{eq:l2} are positive integer multiples of $\deg(v)$, and \eqref{eq:line5}  from \eqref{eq:nonneg} (and the fact that $L_1$ and $L_2$ are sums of local heights relative to $N_e$ points).
From~\eqref{eq:line5} we obtain
\[\left(\frac{N_e-2-\epsilon-N_e/e}{1-1/e}\right)[K:\mathbb{Q}]h(t)\leq L_1(t)+O_\epsilon(1).\]
Note that the coefficient of the left-hand-side can be made positive by taking $\epsilon>0$ small, given our definition of $N_e$.
In other words we obtain $\delta>0$ and $C_\delta$ such that
\begin{equation}\label{eq:abcdelta}\delta [K:\QQ] h(\phi(t))\leq \sum_{i=1}^{N_e}\sum_{\substack{v\not\in S \\ v(t-a_i)>0 \\ e\nmid v(t-a_i)}}[K_v:\QQ_v]\lambda_{[a_i], v}(t)+C_\delta.\end{equation}

Now suppose that $v(t-a_i)>0$ for some $v\not\in S$ and some $i$. By the construction of our set $S$, we have $v(t-a_j)=0$ for any root or pole $a_j$ of $\varphi$ other than $a_i$. It follows that
\[v(\phi(t))=e_iv(t-a_i)\] and,
since $\gcd(e, e_i)=1$, $v(\phi(t))$ cannot be divisible by $e$ unless $v(t-a_i)$ is. If
\[T_i=\{v\not\in S:v(t-a_i)>0\text{ for some $i$ and }e\nmid v(\varphi(t))\},\]
then \eqref{eq:abcdelta} gives 
\begin{eqnarray*}
\delta [K:\QQ] h(\varphi(t))&\leq & \sum_{i=1}^{N_e}\frac{1}{e_i}\sum_{v\in T_i}[K_v:\QQ_v]\lambda_{[\infty], v}(\varphi(t))+C_\delta\\
&\leq & \sum_{v\in T_1\cup\cdots\cup T_{N_e}}[K_v:\QQ_v]\lambda_{[\infty], v}(\varphi(t))+C_\delta,
\end{eqnarray*}
since the $T_i$ are disjoint. 

Now, since $\lambda_{[\infty], v}(x)=\log^+|x|_v$, Lemma~\ref{lem:greens}\eqref{it:typeicontributes} now gives, for any $v\in T=T_1\cup\cdots \cup T_{N_e}$,
\[\frac{1}{d}\lambda_{[\infty], v}(\varphi(t))\leq G_{f_{\varphi(t)}, v}(z),\]
and so
\[\delta h(\varphi(t))\leq \sum_{v\in T}\frac{[K_v:\QQ_v]}{[K:\QQ]}dG_{f_{\varphi(t)}, v}(z) + C_\delta\leq d\hat{h}_{f_{\varphi(t)}}(z)+C_\delta\]
for any $z\in K$, by the non-negativity of $G_{f, w}$ for all $f$ and $w$. This proves the theorem.
 \end{proof}


\section{The proofs of Theorems~\ref{th:dg} and~\ref{th:dm}}\label{sec:darmon}

Before proving Theorem~\ref{th:dg}, we translate a theorem of Darmon and Granville~\cite{darmon-granville} into a statement saying that the denominators of the values of a sufficiently complicated rational function cannot be a non-trival perfect power infinitely often, even modulo a finite set of primes.

\begin{lemma}\label{lem:dglemma}
Let $K$ be a number field, let $S\subseteq M_K$ be a finite set of places, let $e\geq 2$, and let $\varphi(t)\in K(t)$ be a rational function with at least $N_e$ distinct poles of order prime to $e$. There exists a finite $X\subseteq \mathbb{P}^1(K)$ such that for all $t\not\in X$ there exists a non-archimedean $v\not\in S$ with $v(\varphi(t))<0$ and $e\nmid v(\varphi(t))$.
\end{lemma}

\begin{proof}
First note that proving the result for a larger set of primes suffices, and so we will assume without loss of generality that we have
\[\varphi([x:y])=F(x, y)/G(x, y)\]
with $\operatorname{Res}(F, G)\in \mathcal{O}_{K, S}^*$, and that $\mathcal{O}_{K, S}$ is a PID, by adding to $S$ all prime divisors of some set of representatives of the class group. Note that what we are trying to show is now equivalent to the claim that the set of $t\in \mathbb{P}^1(K)$ such that the fractional ideal $\varphi(t)\mathcal{O}_{K, S}$ has the form $I/J^e$, with $I, J\subseteq \mathcal{O}_{K, S}$ integral ideals, is finite.

If we have $x, y\in \mathcal{O}_K$ with $x\mathcal{O}_{K, S}+y\mathcal{O}_{K, S}=\mathcal{O}_{K, S}$ and
\[\frac{F(x, y)}{G(x, y)}\mathcal{O}_{K, S}=\frac{I}{J^e},\]
we in particular have $G(x, y)\mathcal{O}_{K, S}=J^e$. Since $\mathcal{O}_{K, S}$ is a PID, we indeed have $G(x, y)=s_i b^e$ for some $b\in \mathcal{O}_{K, S}$, and $s_i$ a representative of one of the finitely many classes in $\mathcal{O}_{K, S}^\times/(\mathcal{O}_{K, S}^\times)^e$. Clearing denominators of $s_i$ and the coefficients of $G$, we now have
\begin{equation}\label{eq:dg}G_i(x, y)=z^e,\end{equation}
with $z\in \mathcal{O}_K$,
for one of finitely many forms $G_i\in \mathcal{O}_K[x, y]$.

Now, for any $t\in K$, we may write $t=x/y$ with $x, y\in \mathcal{O}_K$ and $x\mathcal{O}_K+y\mathcal{O}_K$ dividing some fixed ideal $I\subseteq\mathcal{O}_K$ with $I\mathcal{O}_{K, S}$ trivial. If infinitely many of these admit a $z$ solving one of the equations \eqref{eq:dg}, then there is one such equation admitting infinitely many solutions, say $G=G_i$. 

Let $S'$ be unordered tuple of values of the form $e/gcd(e, r)$ as $r$ ranges over the orders of the distinct roots of $G(x, 1)$. 
By a theorem of Darmon and Granville~\cite[Theorem~1']{darmon-granville}, if there are infinitely many solutions to~\eqref{eq:dg} satisfying the conditions above then either $S'$ contains one or two entries, or  $S'$ is one of  $\{2, 2, 2, 2\}$, $\{3, 3, 3\}$, $\{2, 4, 4\}$, $\{2, 2, n\}$ for some integer $n$, or $\{2, 3, n\}$ for some integer $3\leq n\leq 6$. Our hypotheses ensure that $e$ occurs in the list at least $N_e$ times, which rules out any of these possibilities.
\end{proof}

\begin{proof}[Proof of Theorem~\ref{th:dg}]
The result follows immediately from Lemma~\ref{lem:greens}\eqref{it:filled} and Lemma~\ref{lem:dglemma}. In particular, by Lemma~\ref{lem:greens} there is a finite set $S$ of places such that if $\alpha\in K$ is a periodic point for $f_{\varphi(t)}$, then $v(\varphi(t))=mv(\alpha)$ for every $v\not\in S$ with $v(\varphi(t))<0$. Lemma~\ref{lem:dglemma} now restricts $t$ to one of finitely many elements of $\mathbb{P}^1(K)$.
\end{proof}

\begin{remark}
Note that our application of the result of Darmon and Granville shows that the results stated here could be strengthened slightly, at the cost of a more complicated formulation. For instance, when $e=4$ we derive a conclusion when $\varphi$ has at least 3 poles of odd order. The same proof gives the same conclusion if $\varphi$ has at least 5 poles or order not divisible by 4 (but not necessarily any of odd order).
\end{remark}

\begin{proof}[Proof of Theorem~\ref{th:dm}]
Let $t=x/y$ with $\gcd(x, y)=1$ and $xy\neq 0$ (the case $xy=0$ corresponds to $f(z)=z^d$). If
\[f(z)=z^d+\frac{x^m}{x^m+y^m}\] has a $\QQ$-rational periodic point other than $z=\infty$ then, by the argument that proves Lemma~\ref{lem:greens}, we have
\[x^m+y^m=w^d\]
for some $w\in \mathbb{Z}$. Specifically, note first that $x^m$ and $x^m+y^m$ are coprime. Now, if $p$ is a prime and $v_p(x^m+y^m)>0$ is not divisible by $d$, then for all $z\in \QQ$ we have $|z|_p^d\neq |t|_p$. Either $|z|_p^d>|t|_p$, in which case an induction shows that $|f^n(z)|_p\to\infty$ or else $|z|^d_p<|t|_p$ in which case $|f(z)|_p^d=|t|_p^d>|t|_p$, in which case $|f^n(z)|_p\to\infty$ again. In other words, if $f_t$ does have a rational preperiodic point (other than $\infty$) and $p\mid (x^m+y^m)$, then $d\mid v_p(x^m+y^m)$. In other words, $x^m+y^m=\pm w^d$ for some $w\in \ZZ$. A negative sign may be asborbed in $w$ if $d$ is odd, into $x$ and $y$ if $m$ is odd, and is impossible of $m$ and $d$ are both even, so without loss of generality we have a solution to $x^m+y^m= w^d$ in integers with $\gcd(x, y)=1$. By the Main Theorem of~\cite{dm}, there is no such solution with $xy\neq 0$.
\end{proof}


\end{document}